\documentclass[12pt]{article}
\usepackage{amssymb}
\usepackage{graphicx}
\usepackage{amsmath}
\usepackage{amsthm}
\newtheorem{theorem}{Theorem}[section]

\newtheorem{lemma}[theorem]{Lemma}

\newtheorem{proposition}[theorem]{Proposition}

\begin{document}

\title{Non-abelian representations of the slim dense near hexagons on 81 and 243 points}
\author{B. De Bruyn, B. K. Sahoo and N. S. N. Sastry}
\maketitle

\begin{abstract}
We prove that the near hexagon $Q(5,2) \times \mathbb{L}_3$ has a non-abelian representation in the extra-special $2$-group  $2^{1+12}_+$ and that the near hexagon $Q(5,2) \otimes Q(5,2)$ has a non-abelian representation in the extra-special $2$-group  $2^{1+18}_-$. The description of the non-abelian representation of $Q(5,2) \otimes Q(5,2)$ makes use of a new combinatorial construction of this near hexagon.
\end{abstract}

\medskip \noindent {\small \textbf{Keywords:} near hexagon, non-abelian representation, extra-special 2-group\\ \textbf{MSC2000:} 05B25}

\section{Introduction} \label{sec1}

Let $\mathcal{S}=(P,L)$ be a partial linear space with point set $P$ and line set $L$. We suppose that $\mathcal{S}$ is {\em slim}, i.e., that every line of $\mathcal{S}$ is incident with precisely three points. For distinct points $x,y \in P$, we write $x \sim y$ if they are collinear. In that case, we denote by $xy$ the unique line containing $x$ and $y$ and define $x*y$ by $xy =\{x,y,x*y \}$. For $x \in P$, we define $x^{\perp} := \{ x \} \cup \{y \in P : y \sim x \}$. If $x,y \in P$, then $d(x,y)$ denotes the distance between $x$ and $y$ in the collinearity graph of $\mathcal{S}$.

A \textit{representation} \cite[p.525]{IPS} of $\mathcal{S}$ is a pair $(R,\psi)$, where $R$ is a group and $\psi$ is a mapping from $P$ to the set of involutions of $R$, satisfying:

(R1) $R$ is generated by the image of $\psi$;

(R2) $\psi$ is one-one on each line $\{ x,y,x*y \}$ of $\mathcal{S}$ and $\psi(x)\psi(y)=\psi(x*y)$.

\noindent Notice that if $x \sim y$, then $\psi(x)$ and $\psi(y)$ necessarily commute by condition (R2). The group $R$ is called a {\it representation group} of $\mathcal{S}$. A representation $(R,\psi)$ of $\mathcal{S}$ is \textit{faithful} if $\psi$ is injective and is \textit{abelian} or \textit{non-abelian} according as $R$ is abelian or not. Note that, in \cite{IPS}, `non-abelian representation' means that `the representation group is not necessarily abelian'. Abelian representations are called embeddings in the literature. For an abelian representation, the representation group is an elementary abelian 2-group and hence can be considered as a vector space over the field $\mathbb{F}_2$ with two elements. We refer to \cite{I} and \cite[Sections 1 and 2]{SS1} for more on representations of partial linear spaces with $p+1$ points per line, where $p$ is a prime.

A finite $2$-group $G$ is called \textit{extra-special} if its Frattini subgroup $\Phi(G)$, its commutator subgroup $G'=[G,G]$ and its center $Z(G)$ coincide and have order $2$. We refer to \cite[Section 20, pp.78--79]{DH} or \cite[Chapter 5, Section 5]{G} for the properties of extra-special 2-groups which we will mention now.  An extra-special $2$-group is of order $2^{1+2m}$ for some integer $m \geq 1$. Let $D_{8}$ and $Q_{8}$, respectively, denote the dihedral and the quaternion groups of order 8. A non-abelian $2$-group of order 8 is extra-special and is isomorphic to either $D_{8}$ or $Q_{8}$. If $G$ is an extra-special 2-group of order $2^{1+2m}$, $m \geq 1$, then the exponent of $G$ is 4 and either $G$ is a central product of $m$ copies of $D_{8}$, or $G$ is a central product of $m-1$ copies of $D_{8}$ and one copy of $Q_{8}$. If the former (respectively, latter) case occurs, then the extra-special 2-group is denoted by $2_{+}^{1+2m}$ (respectively, $2_{-}^{1+2m}$).

A partial linear space $\mathcal{S}=(P,L)$ is called a {\em near polygon} if for every point $p$ and every line $L$, there exists a unique point on $L$ nearest to $p$. If $d$ is the maximal distance between two points of $\mathcal{S}$, then the near polygon is also called a {\em near $2d$-gon}. A near polygon is called {\em dense} if every line is incident with at least three points and if every two points at distance 2 have at least two common neighbours. By \cite{BCHW}, there are up to isomorphism 11 slim dense near hexagons. The paper \cite{SS2} initiated the study of the non-abelian representations of these dense near hexagons.

Suppose $(R,\psi)$ is a non-abelian representation of a slim dense near hex\-ag\-on. Then by \cite[Proposition 4.1, p.205]{SS2}, $(R,\psi)$ necessarily is faithful and for $x,y \in P$, $[\psi(x),\psi(y)] \neq 1$ if and only if $x$ and $y$ are at maximal distance 3 from each other. If $\mathcal{S}$ is the (up to isomorphism) unique slim dense near hexagon on 81 points, which will be denoted by $Q(5,2) \times \mathbb{L}_3$ in the sequel, then it was shown in \cite[Theorem 1.6, p.199]{SS2} that $R$ necessarily is isomorphic to the extra-special 2-group $2^{1+12}_+$. If $\mathcal{S}$ is the (up to isomorphism) unique slim dense near hexagon on 243 points, which will be denoted by $Q(5,2) \otimes Q(5,2)$ in the sequel, then it was shown in \cite[Theorem 1.6, p.199]{SS2}, that $R$ necessarily is isomorphic to the extra-special 2-group $2^{1+18}_-$. The question whether such non-abelian representations exist remained however unanswered in \cite{SS2}. The following theorem, which is the main result of this paper, deals with these existence problems.

\begin{theorem} \label{theo1.1}
$(1)$ The slim dense near hexagon $Q(5,2) \times \mathbb{L}_3$ has a non-abelian representation in the extra-special $2$-group  $2^{1+12}_+$.

$(2)$ The slim dense near hexagon $Q(5,2) \otimes Q(5,2)$ has a non-abelian representation in the extra-special $2$-group  $2^{1+18}_-$.
\end{theorem}

\bigskip \noindent The slim dense near hexagon $Q(5,2) \otimes Q(5,2)$ has many substructures isomorphic to $Q(5,2) \times \mathbb{L}_3$. We will describe a non-abelian representation of $Q(5,2) \times \mathbb{L}_3$ in Section \ref{sec4}. In Section \ref{sec5}, we will use this to construct a non-abelian representation of $Q(5,2) \otimes Q(5,2)$. To describe the non-abelian representation of $Q(5,2) \otimes Q(5,2)$, we make use of a model of $Q(5,2) \otimes Q(5,2)$ which we discuss in Sections \ref{sec2} and \ref{sec3}.

\bigskip \noindent \textbf{Remark.} Two other constructions of non-abelian representations of slim dense near polygons, in particular, of the slim dense near hexagons on 105 and 135 points, can be found in the paper \cite{PS}.

\section{The point-line geometry $\mathcal{S}_{\theta}$} \label{sec2}

Near quadrangles are usually called {\em generalized quadrangles} (GQ's). A GQ is said to be of {\em order} $(s,t)$ if every line is incident with precisely $s+1$ points and if every point is incident with precisely $t+1$ lines. Up to isomorphism, there exist unique GQ's of order $(2,2)$ and $(2,4)$, see e.g. \cite{PT}. These GQ's are denoted by $W(2)$ and $Q(5,2)$, respectively. A {\em spread} of a point-line geometry is a set of lines partitioning its point set. A spread $S$ of $Q(5,2)$ is called a {\em spread of symmetry} if for every line $L \in S$ and every two points $x_1,x_2 \in L$, there exists an automorphism of $Q(5,2)$ fixing each line of $S$ and mapping $x_1$ to $x_2$. By \cite[Section 7.1]{bdb:sos}, $Q(5,2)$ has up to isomorphism a unique spread of symmetry.

\medskip Now, suppose $S$ is a given spread of symmetry of $Q(5,2)$. If $L_1$ and $L_2$ are two distinct lines of $S$ and if $G$ denotes the unique $(3 \times 3)$-subgrid of $Q(5,2)$ containing $L_1$ and $L_2$, then the unique line $L_3$ of $G$ disjoint from $L_1$ and $L_2$ is also contained in $S$.

Suppose $\theta$ is a map from $S \times S$ to $\mathbb{Z}_{3}$ (the additive group of order three) satisfying the following property:
\begin{enumerate}
\item[$(*)$] If $L_{1},L_{2},L_{3}$ are three lines of $S$ contained in
a grid of $Q(5,2)$, then $\theta(L_{1},L_{2})+\theta(L_{2},L_{3})=\theta(L_{1},L_{3})$.
\end{enumerate}
Notice that $\theta(L,L)=0$ and $\theta(M,L)=-\theta(L,M)$ for all $L,M \in S$. With $\theta$, there is associated a point-line geometry $\mathcal{S}_{\theta}$. The points of $\mathcal{S}_{\theta}$ are of four types:
\begin{enumerate}
\item[]
\begin{enumerate}
\item[$(P1)$] The points $x$ of $Q(5,2)$.

\item[$(P2)$] The symbols $\bar{x}$, where $x$ is a point of
$Q(5,2)$.

\item[$(P3)$] The symbols $\bar{\bar{x}}$, where $x$ is a point of $Q(5,2)$.

\item[$(P4)$] The triples $(x,y,i)$, where $i \in \mathbb{Z}_{3}$ and $x,y$
are distinct collinear points of $Q(5,2)$ satisfying $xy \in S$.
\end{enumerate}
\end{enumerate}
The lines of $\mathcal{S}_{\theta}$ are of nine types:
\begin{enumerate}
\item[]
\begin{enumerate}
\item[$(L1)$] The lines $\{x,y,z\}$ of $Q(5,2)$.

\item[$(L2)$] The sets $\{\bar{x},\bar{y},\bar{z}\}$, where $\{x,y,z\}$ is a
line of $Q(5,2)$.

\item[$(L3)$] The sets $\{\bar{\bar{x}},\bar{\bar{y}},\bar{\bar{z}}\}$,
where $\{x,y,z\}$ is a line of $Q(5,2)$.

\item[$(L4)$] The sets $\{x,\bar{x},\bar{\bar{x}}\}$, where $x$ is a
point of $Q(5,2)$.

\item[$(L5)$] The sets $\{ a,(a,b,i),(a,c,i) \}$, where $i \in \mathbb{Z}_{3}$
and $\{ a,b,c \} \in S$.

\item[$(L6)$] The sets $\{ \bar{a},(b,a,i),(c,a,i) \}$, where $i \in \mathbb{Z}_{3}$
and $\{ a,b,c \} \in S$.

\item[$(L7)$] The sets $\{ \bar{\bar{a}},(b,c,i),(c,b,i) \}$, where $i\in \mathbb{Z}_{3}$
and $\{a,b,c\} \in S$.

\item[$(L8)$] The sets $\{(a,b,i),(b,c,j),(c,a,k)\}$, where $\{i,j,k\}=\mathbb{Z}_{3}$
and $\{ a,b,$ $c \}$ is a line belonging to $S$.

\item[$(L9)$] The sets $\{ (a,u,i),(b,v,j),(c,w,k) \}$, where $(i)$ $\{ a,b,c \}$ and $\{ u,v,$ $w \}$ are two disjoint lines of $Q(5,2)$; $(ii)$ $d(a,u)=d(b,v)=d(c,w)$ $=1$; $(iii)$ $au,bv,cw\in S$; $(iv)$ $j=i+\theta(au,bv)$, $k=i+\theta(au,cw)$.
\end{enumerate}
\end{enumerate}
Incidence is containment. One can easily show that $\mathcal{S}_{\theta}$ is a partial linear space. In order to show that two distinct points of $\mathcal{S}_{\theta}$ are contained in at most one line of Type $(L9)$, one has to make use of Property $(*)$.

\section{An isomorphism $Q(5,2)\otimes Q(5,2) \cong \mathcal{S}_\theta$} \label{sec3}

The aim of this section is to show that the slim dense near hexagon $Q(5,2) \otimes Q(5,2)$ is isomorphic to a point-line geometry $\mathcal{S}_\theta$ for a suitable spread of symmetry $S$ of $Q(5,2)$ and a suitable map $\theta: S \times S \to \mathbb{Z}_3$ satisfying Property ($\ast$). We start with recalling some known properties of the near hexagon $Q(5,2) \otimes Q(5,2)$.

(1) Every two points $x$ and $y$ of $Q(5,2) \otimes Q(5,2)$ are contained in a unique convex subspace of diameter 2, called a {\em quad}. The points and lines which are contained in a given quad define a GQ which is isomorphic to either the $(3 \times 3)$-grid or $Q(5,2)$.

(2) If $Q$ is a $Q(5,2)$-quad and $x \not\in Q$, then $x$ is collinear with a unique point $\pi_Q(x) \in Q$ and we denote by $\mathcal{R}_Q(x)$ the unique point of $x \pi_Q(x)$ distinct from $x$ and $\pi_Q(x)$. If $x \in Q$, then we define $\pi_Q(x) = \mathcal{R}_Q(x) := x$. The map $x \mapsto \mathcal{R}_Q(x)$ defines an automorphism of $Q(5,2) \otimes Q(5,2)$. If $Q_1$ and $Q_2$ are two disjoint $Q(5,2)$-quads, then the map $Q_1 \to Q_2; x \mapsto \pi_{Q_2}(x)$ defines an isomorphism between $Q_1$ and $Q_2$.

(3) There exist two partitions $T_1$ and $T_2$ of the point set of $Q(5,2) \otimes Q(5,2)$ into $Q(5,2)$-quads.

(4) Every element of $T_{1}$ intersects every element of $T_{2}$ in a line. As a consequence, $S^{\otimes} := \{Q_{1}\cap Q_{2}:Q_{1}\in T_{1}\text{ and }Q_{2}\in T_{2}\}$ is a spread of $Q(5,2)\otimes Q(5,2)$.

(5) For every $Q \in T_i$, $i \in \{ 1,2 \}$, the set $\{ Q \cap R : R \in T_{3-i} \}$ is a spread of symmetry of $Q$.

(6) Every line $L$ of $Q(5,2)\otimes Q(5,2)$ not belonging to $S^{\otimes}$ is contained in a unique quad of $T_{1} \cup T_2$.

\bigskip Now, let $Q$ and $\overline{Q}$ be two disjoint $Q(5,2)$-quads belonging to $T_{1}$ and put $\overline{\overline{Q}}:=\mathcal{R}_{Q}(\overline{Q})=\mathcal{R}_{\overline{Q}}(Q)$. For every point $x$ of $Q$, put $\bar{x}:=\pi_{\overline{Q}}(x)$ and $\bar{\bar{x}}:=\pi_{\overline{\overline{Q}}}(x)$.

Put $S=\{Q\cap Q_{2}:Q_{2}\in T_{2}\}$. Then $S$ is a spread of symmetry of $Q$. For every $L\in S$, let $R_{L}$ denote the unique element of $T_{2}$ containing $L$. Let $L^{*}$ denote a specific line of $S$ and put $R^{*}=R_{L^{*}}$. For every $L\in S$, $R_{L}\cap (Q \cup \overline{Q} \cup \overline{\overline{Q}})$ is a $(3\times 3)$-subgrid $\sigma_{L}$ of $R_{L}$. This $(3\times 3)$-subgrid $\sigma_{L}$ is contained in precisely three $W(2)$-subquadrangles of $R_{L}$. We denote by $W^{0}, W^{1},W^{2}$ the three $W(2)$-subquadrangles of $R^{*}$
containing $R^{*}\cap (Q \cup \overline{Q} \cup \overline{\overline{Q}})$. For every $L \in S$ and $i\in\mathbb{Z}_{3}$, put $W_{L}^{i}:=\pi_{R_{L}}(W^{i})$.

For every $i\in\mathbb{Z}_{3}$, for every $L\in S$ and for all $x,y \in L$ with $x \neq y$, we denote by $(x,y,i)$ the unique point $\mu$ of $R_{L}\setminus(Q \cup \overline{Q} \cup \overline{\overline{Q}})$ such that $\pi_{Q}(\mu)=x,\pi_{\overline{Q}}(\mu)=\bar{y}$ and $\mu\in W_{L}^{i}$. The point $(x,y,i)$ is the unique point of $W_{L}^{i}$ collinear with $x$ and $\bar{y}$, but not contained in $\sigma_{L}$.

\begin{lemma} \label{lem3.1}
Every point of $Q(5,2) \otimes Q(5,2)$ not contained in $Q \cup \overline{Q} \cup \overline{\overline{Q}}$ has received a unique label.
\end{lemma}
\begin{proof}
Let $\mu$ be a point of $Q(5,2) \otimes Q(5,2)$ not contained in $Q \cup \overline{Q} \cup \overline{\overline{Q}}$, let $R$ denote the unique element of $T_{2}$ containing $\mu$ and put $L:=R\cap Q$. Then $R=R_{L}$. There exists a unique $W(2)$-subquadrangle of $R$ containing $\mu$ and $\sigma_{L}$. Let $i \in \mathbb{Z}_{3}$ such that $\mu \in W_{L}^{i}$. Let $x$ and $y$ be the points of $L$ such that $x = \pi_{Q}(\mu)$ and $\bar{y}=\pi_{\overline{Q}}(\mu)$. If $x=y$, then $\{x,\bar{y},\mu\}$  is a set of mutually collinear points, implying that $\mu=\bar{\bar{x}}$, contradicting $\mu \notin \overline{\overline{Q}}$. Hence $x\neq y$ and the point $\mu$ has label $(x,y,i)$. It is also clear that $\mu$ cannot be labeled in different ways.
\end{proof}

We will now define a map $\theta:S\times S \to \mathbb{Z}_{3}$. For each ordered pair $(L_{1},L_{2})$ of lines of $S$, the map $R^\ast \to R^\ast; x \mapsto \pi_{R^{*}}\circ\pi_{R_{L_{2}}}\circ\pi_{R_{L_{1}}}(x)$
determines an automorphism of $R^{*}$ fixing each line of the spread $\{R^{*}\cap Q_{1}:Q_{1}\in T_{1}\}$ of $R^{*}$. By \cite[Theorem 4.1]{bdb:sos}, such an automorphism either is trivial or acts on any line of the form $R^{*}\cap Q_{1}, Q_{1}\in T_{1}$, as a cycle.  Since every line $R^{*}\cap Q_{1}, Q_{1}\in T_{1}\setminus \{Q,\overline{Q},\overline{\overline{Q}}\}$, intersects each $W(2)$-subquadrangle $W^{i}$, $i \in \mathbb{Z}_{3}$, in a unique point, the map $R^\ast \to R^\ast; x \mapsto \pi_{R^{*}} \circ \pi_{R_{L_{2}}} \circ \pi_{R_{L_{1}}}(x)$ is either trivial or permutes the elements of $\{W^{0}, W^{1}, W^{2}\}$ in
one of the following ways:
$$W^{0}\rightarrow W^{1}\rightarrow W^{2}\rightarrow W^{0},
\;W^{0}\rightarrow W^{2}\rightarrow W^{1}\rightarrow W^{0}.$$
Hence, there exists a unique $\theta(L_{1},L_{2})\in \mathbb{Z}_{3}$ such that
$$\pi_{R^{*}}\circ \pi_{R_{L_{2}}}\circ \pi_{R_{L_{1}}}(W^{i})=W^{i+\theta(L_{1},L_{2})}$$
for every $i \in \mathbb{Z}_{3}$.

\begin{lemma} \label{lem3.2}
The following holds:
\begin{enumerate}
\item[$(i)$] For every $L\in S$, $\theta(L,L)=0$.
\item[$(ii)$] For any two lines $L_{1}$ and $L_{2}$ of $S$, $\theta(L_{2},L_{1})=-\theta(L_{1},L_{2})$.
\item[$(iii)$] If $L_{1},L_{2},L_{3}$ are three lines of $S$ which are contained in a grid, then
$\theta(L_{1},L_{2})+\theta(L_{2},L_{3})=\theta(L_{1},L_{3})$.
\item[$(iv)$] If $L_{1},L_{2},L_{3}$ are three lines of $S$ which are not contained in a grid, then
$\theta(L_{1},L_{2})+\theta(L_{2},L_{3})\neq \theta(L_{1},L_{3})$.
\end{enumerate}
\end{lemma}
\begin{proof}
$(i)$ For every $i\in\mathbb{Z}_{3}$, we have $\pi_{R^{*}}\circ
\pi_{R_{L}}\circ \pi_{R_{L}}(W^{i})=\pi_{R^{*}}\circ
\pi_{R_{L}}(W^{i})=W^{i}$. Hence, $\theta(L,L)=0$.

$(ii)$ If $\pi_{R^{*}}\circ \pi_{R_{L_{2}}}\circ
\pi_{R_{L_{1}}}(W^{i})=W^{i+\theta(L_{1},L_{2})}$ for every $i\in
\mathbb{Z}_{3}$, then $W^{i}=\pi_{R^{*}}\circ \pi_{R_{L_{1}}}\circ
\pi_{R_{L_{2}}}(W^{i+\theta(L_{1},L_{2})})$ for every $i\in
\mathbb{Z}_{3}$. It follows that
$\theta(L_{2},L_{1})=-\theta(L_{1},L_{2})$.

$(iii)$ Let $L_{1},L_{2},L_{3}$ be three lines of $S$ which are
contained in a grid. Then $\pi_{R^{*}}\circ \pi_{R_{L_{3}}}\circ
\pi_{R_{L_{1}}}(W^{i})=\pi_{R^{*}}\circ \pi_{R_{L_{3}}}\circ
\pi_{R_{L_{2}}}\circ \pi_{R_{L_{1}}}(W^{i})=\pi_{R^{*}}\circ
\pi_{R_{L_{3}}}\circ \pi_{R_{L_{2}}}\circ \pi_{R^{*}}\circ
\pi_{R_{L_{2}}}\circ\pi_{R_{L_{1}}}(W^{i})=\pi_{R^{*}}\circ
\pi_{R_{L_{3}}}\circ
\pi_{R_{L_{2}}}(W^{i+\theta(L_{1},L_{2})})=W^{i+\theta(L_{1},L_{2})+\theta(L_{2},L_{3})}$.
Hence,
$\theta(L_{1},L_{3})=\theta(L_{1},L_{2})+\theta(L_{2},L_{3})$.

$(iv)$  Let $L_{1},L_{2},L_{3}$ be three lines of $S$ which are not contained in a grid. Suppose that
$\theta(L_{1},L_{3})=\theta(L_{1},L_{2})+\theta(L_{2},L_{3})$. Then for every $y \in R^\ast$, $\pi_{R^{*}}\circ \pi_{R_{L_{3}}} \circ \pi_{R_{L_{1}}}(y) = (\pi_{R^{*}}\circ \pi_{R_{L_{3}}} \circ
\pi_{R_{L_{2}}})\circ (\pi_{R^{*}}\circ \pi_{R_{L_{2}}}\circ\pi_{R_{L_{1}}})(y)$, that is, $\pi_{R^{*}}\circ
\pi_{R_{L_{3}}}\circ \pi_{R_{L_{1}}}(y)=\pi_{R^{*}}\circ
\pi_{R_{L_{3}}}\circ \pi_{R_{L_{2}}}\circ \pi_{R_{L_{1}}}(y)$. Hence, the map $R_{L_{3}} \to R_{L_{3}}$ defined by $x\mapsto \pi_{R_{L_{3}}}\circ \pi_{R_{L_{2}}}\circ \pi_{R_{L_{1}}}(x)$ is the identity map on $R_{L_{3}}$. This implies that the points $x,\pi_{R_{L_{1}}}(x),\pi_{R_{L_{2}}}(x)$ are mutually collinear
for every $x\in R_{L_{3}}$, that is, $\left\{x,\pi_{R_{L_{1}}}(x),\pi_{R_{L_{2}}}(x)\right\}$ is a line
for every $x\in R_{L_{3}}$. This contradicts the fact that $L_{1},L_{2},L_{3}$ are not contained in a grid. Hence, $\theta(L_{1},L_{3})\neq \theta(L_{1},L_{2})+\theta(L_{2},L_{3})$.
\end{proof}

\begin{proposition} \label{prop3.3}
$Q(5,2)\otimes Q(5,2) \cong \mathcal{S}_{\theta}$, where $\theta$ is as defined above.
\end{proposition}
\begin{proof}
We must show that the set of lines of $Q(5,2)\otimes Q(5,2)$ are in bijective correspondence with the sets of Type $(L1),(L2),\cdots,(L9)$ defined in Section \ref{sec2}. Obviously,
\begin{enumerate}
\item[$\bullet$] the set of lines of $Q(5,2)\otimes Q(5,2)$ contained in $Q$ correspond to the sets of Type $(L1)$;

\item[$\bullet$] the set of lines of $Q(5,2)\otimes Q(5,2)$ contained in $\overline{Q}$ correspond to the sets of Type $(L2)$;
\item[$\bullet$] the set of lines of $Q(5,2)\otimes Q(5,2)$ contained in $\overline{\overline{Q}}$ correspond to the sets of Type $(L3)$;
\item[$\bullet$] the set of lines of $Q(5,2)\otimes Q(5,2)$ meeting $Q,\overline{Q}$ and $\overline{\overline{Q}}$ correspond to the sets of Type $(L4)$.
\end{enumerate}
Consider a line $M$ of $R_{L}$ which is not contained in $\sigma_{L}$ and which intersects $\sigma_{L}$ in a point $a\in L$ of $Q$. Put $L=\{a,b,c\}$. There exists a unique $W(2)$-subquadrangle $W_{L}^{i}$ containing $\sigma_{L}$ and $M$. One readily sees that the points of $M$ have labels $a$, $(a,b,i)$ and $(a,c,i)$. So, $M$ corresponds to a set of Type $(L5)$. Conversely, every set of Type $(L5)$ corresponds to a (necessarily unique) line of $Q(5,2)\otimes Q(5,2)$.

Next, consider a line $M$ of $R_{L}$ which is not contained in $\sigma_{L}$ and which intersects $\sigma_{L}$ in a point $\bar{a}$ of $\overline{Q}$. Put $L=\{a,b,c\}$. Then, there exists a unique $W(2)$-subquadrangle $W_{L}^{i}$ containing $\sigma_{L}$ and $M$. One readily sees that the points of $M$ have labels $\bar{a}$, $(b,a,i)$ and $(c,a,i)$. So, $M$ corresponds to a set of Type $(L6)$. Conversely, every set of Type $(L6)$ corresponds to
a (necessarily unique) line of $Q(5,2)\otimes Q(5,2)$.

Now, consider a line $M$ of $R_{L}$ which is not contained in $\sigma_{L}$ and which intersects $\sigma_{L}$ in a point $\bar{\bar{a}}$ of $\overline{\overline{Q}}$. Put $L=\{a,b,c\}$. Then, there exists a unique $W(2)$-subquadrangle $W_{L}^{i}$ containing $\sigma_{L}$ and $M$. One readily sees that the points of $M$ have labels $\bar{\bar{a}}$, $(b,c,i)$ and $(c,b,i)$. So, $M$ corresponds to a set of Type $(L7)$. Conversely, every set of Type $(L7)$ corresponds to a (necessarily unique) line of $Q(5,2) \otimes Q(5,2)$.

Consider next a line $M$ of $R_{L}$ which is disjoint from $\sigma_{L}$. Then $M$ intersects each $W_{L}^{i}$,
$i\in\mathbb{Z}_{3}$, in a unique point. Put $L=\{a,b,c\}$. The labels of the points of $M$ are $(u,u',i),\;(v,v',j),\;(w,w',k)$, where $\{ i,j,k \} = \{ 0,1,2 \}$, $\{u,v,w\}=\pi_{Q}(M)=\{a,b,c\}$, $\{u',v',w'\}=\pi_{Q}\circ\pi_{\overline{Q}}(M)=\{a,b,c\}$, $u\neq u',\;v\neq v',\;w\neq w'$. It readily follows that $\{(u,u',i),(v,v',j),(w,w',k)\}$ is a set of Type $(L8)$.
Conversely, one can readily verify that every set of Type $(L8)$ corresponds to a line of $Q(5,2)\otimes Q(5,2)$.

Finally, let $M$ be a line of $Q(5,2)\otimes Q(5,2)$ not belonging to $S^{\otimes}$ and contained in a quad of $T_{1}\setminus \{Q,\overline{Q},\overline{\overline{Q}}\}$. With $M$, there corresponds a set of the form $\{(a,u,i),(b,v,j),(c,w,k)\}$. We have that $\{a,b,c\}=\pi_{Q}(M)$ is a line of $Q$ not belonging to $S$.
Similarly, $\{u,v,w\}=\pi_{Q}\circ\pi_{\overline{Q}}(M)$ is a line of $Q$ not belonging to $S$. Moreover, we have that $au,bv,cw\in S$ and $j=i+\theta(au,bv)$, $k=i+\theta(au,cw)$ by the definition of the map $\theta$. So, $M$ corresponds to a set of Type $(L9)$. Conversely, we show that every set $\{(a,u,i),(b,v,j),(c,w,k)\}$ of Type $(L9)$ corresponds to a line of $Q(5,2)\otimes Q(5,2)$ not belonging to $S^{\otimes}$ and contained in a quad of $T_{1}\setminus
\{Q,\overline{Q},\overline{\overline{Q}}\}$. Let $x$ denote the point of $Q(5,2)\otimes Q(5,2)$ corresponding to $(a,u,i)$, let $Q_{1}$ denote the unique element of $T_{1}$ containing $x$ and let $M=\pi_{Q_{1}}(\{a,b,c\})$. Then $M$ corresponds to a set of the form $\{(a,u,i),(b,*,*),(c,*,*)\}$. Since $v,w,j,k$ are uniquely determined by $a,u,i,b,c$, this set is equal to $\{ (a,u,i),(b,v,j),(c,w,k) \}$.

By the above discussion, we indeed know that $Q(5,2)\otimes Q(5,2) \cong \mathcal{S}_{\theta}$.
\end{proof}

\bigskip \noindent \textbf{Definitions}. (1) An {\em admissible triple} is a triple $\Sigma = (\mathcal{L},G,\Delta)$, where:

$\bullet$ $G$ is a nontrivial additive group whose order $s+1$ is finite.

$\bullet$ $\mathcal{L}$ is a linear space, different from a point, in which each line is incident with exactly $s+1$ points. We denote the point set of $\mathcal{L}$ by $P$.

$\bullet$ $\Delta$ is a map from $P \times P$ to $G$ such that the following holds for any three points $x$, $y$ and $z$ of $\mathcal{L}$: $x$, $y$ and $z$ are collinear $\Leftrightarrow$ $\Delta(x,y) + \Delta(y,z) = \Delta(x,z)$.

(2) Suppose $\Sigma_1 = (\mathcal{L}_1,G_1,\Delta_1)$ and $\Sigma_2 = (\mathcal{L}_2,G_2,\Delta_2)$ are two admissible triples, where $\mathcal{L}_1$ and $\mathcal{L}_2$ are not lines. Then $\Sigma_1$ and $\Sigma_2$ are called {\em equivalent} if there exists an isomorphism $\alpha$ from $\mathcal{L}_1$ to $\mathcal{L}_2$, an isomorphism $\beta$ from $G_1$ to $G_2$ and a map $f$ from the point set of $\mathcal{L}_1$ to $G_1$ satisfying $\Delta_2(\alpha(x),\alpha(y)) = (f(x)+\Delta_1(x,y)-f(y))^\beta$ for all points $x$ and $y$ of $\mathcal{L}_1$.

\bigskip \noindent Let $\mathcal{L}_S$ denote the linear space whose points are the elements of $S$ and whose lines are the unordered triples of lines of $S$ which are contained in a grid, with incidence being containment. Then $\mathcal{L}_S$ is isomorphic to the affine plane $\mathrm{AG}(2,3)$ of order three. By Lemma \ref{lem3.2}, we know that $(\mathcal{L}_S,\mathbb{Z}_{3},\theta)$ is an admissible triple.

\begin{proposition} \label{prop3.4}
Let $\theta_{1}$ and $\theta_{2}$ be two maps from $S\times S$ to $\mathbb{Z}_{3}$ such that
$\Sigma_{1}=(\mathcal{L}_S,\mathbb{Z}_{3},\theta_{1})$ and $\Sigma_{2}=(\mathcal{L}_S,\mathbb{Z}_{3},\theta_{2})$ are
admissible triples. If $\Sigma_{1}$ and $\Sigma_{2}$ are equivalent, then $\mathcal{S}_{\theta_{1}} \cong \mathcal{S}_{\theta_{2}}$.
\end{proposition}
\begin{proof}
Since $\Sigma_{1}$ and $\Sigma_{2}$ are equivalent, there exists an automorphism $\alpha$ of $\mathcal{L}_S$, an automorphism $\beta$ of $\mathbb{Z}_{3}$ and a map $f$ from $S$ to $\mathbb{Z}_{3}$ satisfying $\theta_{2}(\alpha(x),\alpha(y))=(f(x)+\theta_{1}(x,y)-f(y))^\beta$ for all points $x$ and $y$ of $\mathcal{L}_S$. There exists an automorphism $\phi$ of $Q$ such that $\alpha(L)=\phi(L)$ for every line $L$ of $S$, see e.g. \cite[Section 3, Example 1]{bdb:number}. One readily verifies that the map $x \mapsto x^\phi; \bar{x} \mapsto \overline{x^{\phi}}; \bar{\bar{x}} \mapsto
\overline{\overline{x^{\phi}}}; (a,b,i) \mapsto (a^{\phi},b^{\phi},(i-f(ab))^\beta)$ defines an isomorphism between
$\mathcal{S}_{\theta_1}$ and $\mathcal{S}_{\theta_2}$.
\end{proof}

\bigskip \noindent It is known that the affine plane $\mathrm{AG}(2,3)$ admits, up to equivalence, a unique admissible
triple. [This follows, for instance, from \cite[Theorem 2.1]{DB-T} and the fact that there exists a unique generalized quadrangle of order $(2,4)$, namely $Q(5,2)$, and a unique spread of symmetry in $Q(5,2)$.] If we coordinatize $\mathrm{AG}(2,3)$ in the standard way, then an admissible triple can be obtained by putting $\Delta[(x_1,y_1),(x_2,y_2)] := x_1 y_2 - x_2 y_1 \in \mathbb{Z}_3$.

\section{A non-abelian representation of the near hexagon $Q(5,2) \times \mathbb{L}_{3}$} \label{sec4}

The slim dense near hexagon $Q(5,2) \times \mathbb{L}_3$ is obtained by taking three isomorphic copies of $Q(5,2)$ and joining the corresponding points to form lines of size 3. In this section we prove that there exists a non-abelian representation of $Q(5,2) \times \mathbb{L}_{3}$.

Let $Q$ and $B$, respectively, be the point and line set of $Q(5,2)$. Set $\overline{Q}=\{\bar{x}:x\in Q\}$,
$\overline{\overline{Q}}=\{\bar{\bar{x}}:x\in Q\}$, $\overline{B}=\{\{\bar{x},\bar{y},\bar{z}\}: \{x,y,z\}\in B\}$ and $\overline{\overline{B}}=\{\{\bar{\bar{x}},\bar{\bar{y}},\bar{\bar{z}}\}: \{x,y,z\}\in B\}$. Then $(\overline{Q},\overline{B})$ and $(\overline{\overline{Q}},\overline{\overline{B}})$ are isomorphic to $Q(5,2)$. The near hexagon $Q(5,2)\times \mathbb{L}_{3}$ is isomorphic to the geometry whose point set $P$ is $Q \cup \overline{Q} \cup \overline{\overline{Q}}$ and whose line set $L$ is $B \cup \overline{B} \cup \overline{\overline{B}} \cup \{\{x,\bar{x},\bar{\bar{x}}\}:x\in Q\}$.

It is known that if $Q(5,2)\times \mathbb{L}_{3}$ admits a non-abelian representation, then the representation group must be the extra-special 2-group $2_{+}^{1+12}$ \cite[Theorem 1.6, p.199]{SS2}. Let $R=2_{+}^{1+12}$ with $R'=\{1,\lambda\}$. Set $V=R/R'$. Consider $V$ as a vector space over $\mathbb{F}_{2}$. The map $f:V\times V \to \mathbb{F}_{2}$ defined by
\begin{equation*}
f(xR',yR')=\left\{\begin{array}{ll}
  0 & \text{ if }[x,y]=1 \\
  1  & \text{ if }[x,y]=\lambda
\end{array}\right.
\end{equation*}
for $x,y\in R$, is a non-degenerate symplectic bilinear form on $V$ \cite[Theorem 20.4, p.78]{DH}. Write $V$ as an orthogonal direct sum of six hyperbolic planes $K_{i}$ ($1\leq i\leq 6$) in $V$ and let $H_{i}$ be the inverse image of $K_{i}$ in $R$ (under the canonical homomorphism $R \to R/R'$). Then each $H_{i}$ is generated by two involutions $x_{i}$ and $y_{i}$ such that $[x_{i},y_{i}] =\lambda$. Let $M=\langle x_{i}: 1\leq i\leq 6\rangle$ and $\overline{M}=\langle y_{i}: 1\leq i\leq 6\rangle$. Then $M$ and $\overline{M}$ are elementary abelian 2-subgroups of $R$ each of order $2^{6}$. Further, $M,\overline{M}$ and $Z(R)$ pairwise intersect trivially and $R = M \overline{M} Z(R)$. Also, $C_{M}(\overline{M})$ and $C_{\overline{M}}(M)$ are trivial.

We regard the points and lines of $Q$ as the points and lines of a nonsingular elliptic quadric of the projective space $\mathrm{PG}(M)$, where $M$ is regarded as a 6-dimensional vector space over $\mathbb{F}_2$. Let $(M,\tau)$ be the natural abelian representation of $(Q,B)$ associated with this embedding of $Q$ in $\mathrm{PG}(M)$. For every point $x$ of $Q$, put $m_x = \tau(x)$. There exists a unique non-degenerate symplectic bilinear form $g$ on $M$ such that $m_x^\perp = \langle m_y: y \in x^{\perp} \rangle$ for every point $x$ of $Q$, see e.g. \cite[Section 22.3]{HT}. Here, the following notational convention has been used: for every $m \in M$, $m^\perp$ denotes the set of all $m' \in M$ for which $g(m,m')=0$.

Now, let $m$ be an arbitrary element of $M$. If $m=1$, then we define $\overline{m} : = 1$. Suppose now that $m \not= 1$. Then $m^\perp$ is maximal in $M$, that is, of index 2 in $M$. So, the centralizer of $m^\perp$ in $\overline{M}$ is a subgroup $\langle \overline{m} \rangle$ of order 2. Since $m^\perp$ is maximal in $M$, $\langle m^\perp,m' \rangle = M$ for every $m' \in M \setminus m^\perp$. The triviality of $C_{\overline{M}}(M)$ then implies that $[\overline{m},m']=\lambda$ for every $m' \in M \setminus m^\perp$.

We prove that the map $M \to \overline{M}; m \mapsto \overline{m}$ is an isomorphism. This map is easily seen to be bijective. (Notice that $C_M(\overline{m}) = m^\perp$.) So, it suffices to prove that $\overline{m_1 m_2} = \overline{m_1} \ \overline{m_2}$ for all $m_1,m_2 \in M$. Clearly, this holds if $1 \in \{ m_1,m_2 \}$ or $m_1 = m_2$. So, we may suppose that $m_1 \not= 1 \not= m_2 \not= m_1$. The set $\{ m_1,m_2,m_1m_2 \}$ corresponds to a line of $\mathrm{PG}(M)$. So, for every $m \in (m_1m_2)^\perp$, $([\overline{m_1},m],[\overline{m_2},m])$ is equal to either $(1,1)$ or $(\lambda,\lambda)$. Then
$$[\overline{m_1} \ \overline{m_2},m]=[\overline{m_1},m][\overline{m_2},m]=1.$$
The first equality holds since $R$ has nilpotency class 2. Thus $\overline{m_1} \ \overline{m_2} \in C_{\overline{M}}((m_1m_2)^\perp) = \langle \overline{m_1m_2} \rangle$. Since $\overline{m_1} \ \overline{m_2} \not= 1$, we have $\overline{m_1} \ \overline{m_2} = \overline{m_1m_2}$.

We conclude that if we define $\overline{\tau}: \overline{Q} \to \overline{M}; \bar{x} \mapsto \overline{m_x}$ for every $x \in Q$, then $(\overline{M},\overline{\tau})$ is a faithful abelian representation of $(\overline{Q},\overline{B})$.

Now, let $m$ be an arbitrary element of $M$. If $m=1$, then we define $\overline{\overline{m}} : = 1$. If $m = m_x$ for some $x \in Q$, then we define $\overline{\overline{m}} := m \overline{m}$. If $m \not= 1$ and $m \not= m_x$, $\forall x \in Q$, then we define $\overline{\overline{m}} := m \overline{m} \lambda$. Since $m^2=\overline{m}^2=\lambda^2=[m,\overline{m}]=1$, $\overline{\overline{m}}$ is an involution. We prove that the map $m \mapsto \overline{\overline{m}}$ defines an isomorphism between $M$ and an elementary abelian 2-group $\overline{\overline{M}}$ of order $2^6$. Since $R = M \overline{M} Z(R)$, this map is injective and hence it suffices to prove that $\overline{\overline{m_1 m_2}} = \overline{\overline{m_1}} \ \overline{\overline{m_2}}$ for all $m_1,m_2 \in M$. Obviously, this holds if $1 \in \{ m_1,m_2 \}$ or $m_1=m_2$. So, we may suppose that $m_1 \not= 1 \not= m_2 \not= m_1$. The set $\{ m_1,m_2,m_1m_2 \}$ corresponds to a line of $\mathrm{PG}(M)$. Suppose $3-N$ elements of $\{ m_1,m_2,m_1m_2 \}$ correspond to points of $Q$. Then $m_1 \in m_2^\perp$ if and only if $N$ is even\footnote{Perhaps the case $N=3$ needs more explanation. Suppose $N=3$ and $m_1 \in m_2^\perp$. Then the hyperplane $\pi$ of $\mathrm{PG}(M)$ corresponding to $m_2^\perp$ intersects $Q$ in a nonsingular parabolic quadric $Q(4,2)$ of $\pi$. Since the point of $\mathrm{PG}(M)$ corresponding to $m_2$ is the kernel of $Q(4,2)$, the line of $\mathrm{PG}(M)$ corresponding to $\{ m_1,m_2,m_1m_2 \} \subset \pi$ must meet $Q(4,2)$, in contradiction with $N=3$.}. So, $[\overline{m_1},m_2]=\lambda^N$. If $N'$ is the number of elements of $\{ m_1,m_2 \}$ corresponding to points of $Q$, then $2-N'-N \in \{ -1,0 \}$ and $2-N'-N=0$ if and only if $m_1m_2$ corresponds to a point of $Q$. Hence, $\overline{\overline{m_1}} \ \overline{\overline{m_2}} = m_1 \ \overline{m_1} \ m_2 \ \overline{m_2} \ \lambda^{2-N'} = m_1 \ m_2 \ \overline{m_1} \ \overline{m_2} \ \lambda^{2-N'-N} = m_1m_2 \ \overline{m_1m_2} \  \lambda^{2-N'-N} = \overline{\overline{m_1m_2}}$.

So, if we define $\overline{\overline{\tau}}:\overline{\overline{Q}} \to \overline{\overline{M}}$ by putting $\overline{\overline{\tau}}(\bar{\bar{x}}) := \overline{\overline{m_x}} = m_x \overline{m_x}$ for all $x \in Q$, then $(\overline{\overline{M}},\bar{\bar{\tau}})$ is a faithful abelian representation of $(\overline{\overline{Q}},\overline{\overline{B}})$.

Now, define a map $\psi: P \to R$ which coincides with $\tau$ on $Q$, $\overline{\tau}$ on $\overline{Q}$ and $\overline{\overline{\tau}}$ on $\overline{\overline{Q}}$. Since $R = \langle M,\overline{M} \rangle$, $R = \langle \psi(P) \rangle$. By construction, $(R,\psi)$ also satisfies Property (R2) in the definition of representation. Hence, $(R,\psi)$ is a non-abelian representation of $Q(5,2) \times \mathbb{L}_{3}$.

\section{A non-abelian representation of the near hexagon $Q(5,2)\otimes Q(5,2)$} \label{sec5}

In this section, we prove that the slim dense near hexagon $Q(5,2) \otimes Q(5,2)$ has a non-abelian representation. By Proposition \ref{prop3.3}, this is equivalent with showing that the partial linear space $\mathcal{S}_\theta$ has a non-abelian representation, where $\theta$ is as defined in Section \ref{sec3}.

We continue with the notation introduced in Section \ref{sec3}. Let $M^\ast$ be a line of $S^\otimes$ contained in $R^\ast$ but distinct from $R^\ast \cap Q$, $R^\ast \cap \overline{Q}$ and $R^\ast \cap \overline{\overline{Q}}$. Then $M^\ast$ intersects each $W^i$, $i \in \{ 1,2,3 \}$, in a unique point. For every point $x$ of $L^\ast$, put $\epsilon(x) := i$ if the unique point of $M^\ast$ collinear with $x$ belongs to $W^i$. If $y \in Q \setminus L^\ast$, then we define $\epsilon(y) := \epsilon(x)$, where $x$ is the unique point of $L^\ast$ collinear with $y$.

\begin{lemma} \label{lem5.1}
Let $L_{1}$ and $L_{2}$ be two distinct lines in $S$ and let $\alpha_i \in L_i$, $i \in \{ 1,2 \}$. Then $\alpha_1 \sim \alpha_2$ if and only if $\epsilon(\alpha_2)-\epsilon(\alpha_1)=\theta(L_1,L_2)$.
\end{lemma}
\begin{proof}
Let $\alpha_2'$ be the unique point of $L_2$ collinear with $\alpha_1$, let $x_1$ and $x_2$ be the unique points of $L^\ast$ nearest to $\alpha_1$ and $\alpha_2'$, respectively, and let $z_i$, $i \in \{ 1,2 \}$, denote the unique point of $M^\ast$ collinear with $x_i$. The automorphism $R^\ast \to R^\ast; x \mapsto \pi_{R^\ast} \circ \pi_{R_{L_2}} \circ \pi_{R_{L_1}}(x)$ of $R^\ast$ maps $x_1$ to $x_2$ and hence $z_1$ to $z_2$. This implies that $W^{\epsilon(x_1) + \theta(L_1,L_2)} = W^{\epsilon(x_2)}$, i.e. $\theta(L_1,L_2) = \epsilon(x_2) - \epsilon(x_1) = \epsilon(\alpha_2') - \epsilon(\alpha_1)$. Hence, $\alpha_1 \sim \alpha_2$ if and only if $\alpha_2 = \alpha_2'$, i.e. if and only if $\epsilon(\alpha_2)-\epsilon(\alpha_1)=\theta(L_1,L_2)$.
\end{proof}

\begin{lemma} \label{lem5.2}
Let $N=2_{-}^{1+6}$ with $N'=\{1,\lambda\}$ and let $I_{2}(N)$ be the set of involutions in $N$. Then there exists a map $\delta$ from $Q$ to $I_{2}(N)$ satisfying the following:
\begin{enumerate}
\item[$(i)$] $\delta $ is one-one.
\item[$(ii)$] For $x,y\in Q$, $[\delta(x),\delta (y)]=1$ if and only if $y\in x^{\perp}$.
\item[$(iii)$] If $x,y\in Q$ with $x\sim y$, then
\begin{equation*}
\delta(x\ast y)=\left\{\begin{array}{ll}
  \delta(x)\delta(y) & \text{ if }xy\in S \\
  \delta(x) \delta (y)\lambda & \text{ if }xy\notin S
\end{array}.\right.
\end{equation*}
\item[$(iv)$] The image of $\delta$ generates $N$.
\end{enumerate}
\end{lemma}
\begin{proof}[\bf Proof.]
We use a model for the generalized quadrangle $Q \cong Q(5,2)$ which is described in \cite[Section 6.1, pp.101--102]{PT}. Put $\Omega = \{1,2,3,4,5,6\}$ and $\Omega' = \{1',2',3',4',5',6'\}$. Let $\mathcal{E}$ be the set of all 2-subsets of $\Omega$ and let $\mathcal{F}$ be the set of all partitions of $\Omega$ in three 2-subsets of $\Omega$. Then the point set of $Q$ can be identified with the set $\mathcal{E} \cup \Omega \cup \Omega'$ and the line set of $Q$ can be identified with the set $\mathcal{F} \cup \{ \{ i,\{ i,j \},j' \} : 1 \leq i , j \leq 6, \ i \neq j \}$. Now, consider the following nine lines of $Q$:
\begin{eqnarray*}
L_{1} = \{ \{ 1,2 \},\{ 3,4 \},\{ 5,6 \} \}; \text{ } L_{2}=\{ \{ 1,4 \},1,4' \}; \text{ }L_{3}=\{\{2,6\},2,6' \};\\
L_{4}=\{\{1,6\},\{2,4\},\{3,5\} \}; \text{ }L_{5}= \{\{1,5\},1',5 \}; \text{ } L_{6}=\{\{2,3\},2',3 \};\\
L_{7}= \{\{1,3\},\{2,5\},\{4,6\} \};\text{ }L_{8}= \{\{3,6\},3',6 \}; \text{ }L_{9}= \{\{4,5\},4,5' \}.
\end{eqnarray*}
These 9 lines are mutually disjoint and hence determine a spread $S'$ of $Q$. Any two distinct lines $L_i$ and $L_j$ of $S'$ are contained in a unique $(3 \times 3)$-subgrid and the unique line of this subgrid disjoint from $L_i$ and $L_j$ also belongs to $S'$. A spread of $Q(5,2)$ having this property is called {\em regular}. Since any regular spread of $Q(5,2)$ is also a spread of symmetry \cite[Section 7.1]{bdb:sos}, and there exists up to isomorphism a unique spread of symmetry in $Q(5,2)$, we may without loss of generality suppose that $S = S'$.

Put $N= \langle a,b \rangle \circ \langle c,d \rangle \circ Q_{8}$, where $a,b,c,d$ are involutions and $\langle a,b \rangle \cong \langle c,d \rangle \cong D_{8}$. So, $[a,b]=[c,d]=\lambda$. Take $Q_{8}= \{1,\lambda,i,j,k,i \lambda,j \lambda,k \lambda \}$, where $i^{2}=j^{2}=k^{2}=\lambda$, $ij=k$, $jk=i$, $ki=j$ and $[i,j]=[j,k]=[k,i]=\lambda$. We define $\delta: Q \to I_{2}(N)$ as follows:
\begin{enumerate}
\item[] $\delta(\{ 1,2 \})=a$, $\delta(\{ 3,4 \})=c$, $\delta(\{ 5,6 \})=ac$,\\
$\delta(\{ 1,4 \})=abdi$, $\delta(1)=cdj$, $\delta(4')=abck\lambda$,\\
$\delta(\{ 2,6 \})=abi\lambda$, $\delta(2)=acdk$, $\delta(6')=bcdj\lambda$,\\
$\delta(\{ 1,6 \})=b$, $\delta(\{ 2,4 \})=bd$, $\delta(\{ 3,5 \})=d$,\\
$\delta(\{ 1,5 \})=abci$, $\delta(1')=cdk\lambda$, $\delta(5)=abdj$,\\
$\delta(\{ 2,3 \})=bcdi\lambda$, $\delta(2')=acdj\lambda$, $\delta(3)=abk$,\\
$\delta(\{ 1,3 \})=abcd\lambda$, $\delta(\{ 2,5 \})=bc\lambda$, $\delta(\{ 4,6 \})=ad\lambda$,\\
$\delta(\{ 3,6 \})=acdi\lambda$, $\delta(3')=abj\lambda$, $\delta(6)=bcdk$,\\
$\delta(\{ 4,5 \})=cdi$, $\delta(4)=abcj$, $\delta(5')=abdk\lambda$.
\end{enumerate}
Put $W = N/N'$. Suppose $\{ x_1,x_2,\ldots,x_6 \}$ is a set of 6 points of $Q$ such that the smallest subspace $[x_1,x_2,\ldots,x_6]$ of $Q$ containing $\{ x_1,x_2,\ldots,x_6 \}$ coincides with $Q$. If $\tau$ is an abelian representation of $Q$ in $W$, then by Property (R1) in the definition of representation, $W = \langle \tau(x_1),\ldots,\tau(x_6) \rangle$ and hence $\{ \tau(x_1),\ldots,\tau(x_6) \}$ is a basis of $W$ (regarded as $\mathbb{F}_2$-vector space). Conversely, if $\{ w_1,\ldots,w_6 \}$ is a basis of $W$, then the map $x_i \mapsto w_i$, $i \in \{ 1,\ldots,6 \}$, can be extended to a unique abelian representation $\tau$ of $Q$ in $W$. (Since there exists an abelian representation of $Q$ in $W$, there must exist an abelian representation $\tau$ for which $\tau(x_i)=w_i$, $i \in \{ 1,\ldots,6 \}$. The uniqueness of $\tau$ follows from the fact that $\tau(y_1 \ast y_2) = \tau(y_1) \tau(y_2)$ for any two distinct collinear points $y_1$ and $y_2$ of $Q$.) Consider now the special case where $x_1 = \{ 1,2  \}$, $x_2 = \{ 3,4 \}$, $x_3 = \{ 3,5 \}$, $x_4 = \{ 1,6 \}$, $x_5 = \{ 4,5 \}$, $x_6=1$, $w_1=a N'$, $w_2 = c N'$, $w_3=d N'$, $w_4 = b N'$, $w_5 = cdi N'$ and $w_6 = cdj N'$. One indeed readily verifies that $[x_1,\ldots,x_6]=Q$ and that $\{ w_1,\ldots,w_6 \}$ is a basis of the $\mathbb{F}_2$-vector space $W$. Let $\delta'$ denote the unique abelian representation of $Q$ in $W$ for which $\delta'(x_i)=w_i$, $i \in \{ 1,\ldots,6 \}$. Then, using the fact that $\delta'(y_1 \ast y_2) = \delta'(y_1) \delta'(y_2)$ for any two distinct collinear points $y_1$ and $y_2$ of $Q$, one can verify that $\delta'(y) = \delta(y) N'$ for every $y \in Q$. This implies that $\delta(y_1 \ast y_2)$ is equal to either $\delta(y_1) \delta(y_2)$ or $\delta(y_1) \delta(y_2) \lambda$ for any two distinct collinear points $y_1$ and $y_2$ of $Q$.

Clearly, the map $\delta: Q \to I_2(N)$ satisfies the properties (i) and (iv) of the lemma. We will now prove that also property (ii) of the lemma is satisfied. So, if $\{ y_1,y_2 \}$ is one of the 351 unordered pairs of distinct points of $Q$, then we need to prove that $[\delta(y_1),\delta(y_2)]=1$ if and only if $y_1 \in y_2^\perp$. Since $Q=[x_1,\ldots,x_6]$, it suffices to prove the following three statements: (I) the above claim holds if $\{ y_1,y_2 \} \subseteq \{ x_1,\ldots,x_6 \}$; (II) if $[\delta(y_1),\delta(y_2)]=1$ for some distinct collinear points $y_1$ and $y_2$, then also $[\delta(y_1),\delta(y_1 \ast y_2)]=1$; (III) if the above claim holds for unordered pairs $\{ y_1,y_2 \}$ and $\{ y_1,y_3 \}$ of points where $y_2 \sim y_3$ and $y_1 \not\in y_2y_3$, then it also holds for the unordered pair $\{ y_1,y_2 \ast y_3 \}$. Statement (I) is easily verified by considering all 15 pairs $\{ x_i,x_j \}$ where $i,j \in \{ 1,\ldots,6 \}$ with $i \not= j$. As to Statement (II), notice that $[\delta(y_1),\delta(y_1 \ast y_2)]$ is equal to either $[\delta(y_1),\delta(y_1) \delta(y_2)]$ or  $[\delta(y_1),\delta(y_1) \delta(y_2) \lambda]$ which is in any case equal to 1. We now prove Statement (III). Since $\delta(y_2 \ast y_3)$ is equal to either $\delta(y_2) \delta(y_3)$ or $\delta(y_2) \delta(y_3) \lambda$, we have $[\delta(y_1),\delta(y_2 \ast y_3)] = [\delta(y_1),\delta(y_2) \delta(y_3)] = [\delta(y_1),\delta(y_2)] [\delta(y_1),\delta(y_3)]$. If $y_1$ is collinear with precisely one of $y_2,y_3$, then $y_1$ is not collinear with $y_2 \ast y_3$ and $[\delta(y_1),\delta(y_2 \ast y_3)] = [\delta(y_1),\delta(y_2)] [\delta(y_1),\delta(y_3)] = 1 \cdot \lambda = \lambda$. If $y_1$ is collinear with $y_2 \ast y_3$, then $[\delta(y_1),\delta(y_2 \ast y_3)] = [\delta(y_1),\delta(y_2)] [\delta(y_1),\delta(y_3)] = \lambda \cdot \lambda = 1$. So, this proves Statement (III) and finishes the proof of property (ii) of the lemma.

Property (iii) of the lemma is verified by considering all 45 lines $L$ of $Q$ and an ordered pair $(x,y)$ of distinct points of $L$. Notice that by property (ii) of the lemma, we only need to consider one ordered pair $(x,y)$ for each line $L$ of $Q$.
\end{proof}

It is known that if the near hexagon $Q(5,2)\otimes Q(5,2)$ admits a non-abelian representation, then the representation group must be the extra-special 2-group $2_{-}^{1+18}$ \cite[Theorem 1.6, p.199]{SS2}. We next construct a non-abelian representation of $\mathcal{S}_{\theta}\cong Q(5,2)\otimes Q(5,2)$ in the group
$2_{-}^{1+18}$.

Let $R=2_{-}^{1+18}$ with $R'=\{1,\lambda\}$. Write $R$ as a central product $R=M \circ N$, where $M=2_{+}^{1+12}$ and $N=2_{-}^{1+6}$. Let $Y= Q \cup \overline{Q} \cup \overline{\overline{Q}}$. Then the subgeometry of $\mathcal{S}_{\theta}$ whose point set is $Y$ together with the lines of types $(L1)-(L4)$ is isomorphic to
$Q(5,2) \times \mathbb{L}_{3}$. Let $P$ be the point set of $\mathcal{S}_{\theta}$ and let $\delta$ be a map from $Q$ to $I_{2}(N)$ satisfying the conditions of Lemma \ref{lem5.2}. We extend $\delta$ to the set $P \setminus Y$ using the map $\epsilon: Q \to \mathbb{Z}_{3}$ which we defined in the beginning of this section:
\begin{quote}
For $L_{1}\in S$, distinct points $a,b \in L_{1}$ and $j\in \mathbb{Z}_{3}$, we define $\delta(a,b,j) := \delta(u)$, where $u$ is the unique point of $L_1$ with $\epsilon(u)=j$.
\end{quote}
Now, fix a non-abelian representation $(M,\phi)$ of $Y$. Such a representation exists by Section \ref{sec4}. Let $\psi$ be the following map from $P$ to $R$:
\begin{enumerate}
\item[$\bullet$] if $q\in Y$, then $\psi(q) := \phi(q)$;
\item[$\bullet$] if $q=(a,b,i) \in P \setminus Y$, then $\psi(q) = \psi(a,b,i) := \phi(b) \phi(\bar{a}) \delta(a,b,i)$.
\end{enumerate}
We prove the following.

\begin{theorem} \label{theo5.3}
$(R,\psi)$ is a non-abelian representation of $\mathcal{S}_{\theta}$.
\end{theorem}
\begin{proof}[\bf Proof.]
Since the image of $\phi$ generates $M$ and the image of $\delta$ generates $N$, we have $R=\langle\psi(P)\rangle$. For every line
$L_{1}\in S$ and distinct $a,b\in L_{1}$, we have
$[\phi(a),\phi(\bar{b})]=1$, since $a$ and $\bar{b}$ are in
distance two from each other. This implies that $\psi(q)$ is an
involution for every $q\in P$. We need to verify condition
$(R2)$ in the definition of representation. This is true
for all lines of types $(L1)-(L4)$, since they are also lines of $Y$ and $\psi$ coincides with $\phi$ on $Y$.

Let $\{a,(a,b,i),(a,c,i)\}$ be a line of type $(L5)$. Since
$\delta(a,b,i)=\delta(a,c,i)$, we have $\psi(a,b,i)\psi(a,c,i)
=\phi(b)\phi(\bar{a})\phi(c)\phi(\bar{a})=\phi(b)\phi(c)=\phi(a)=\psi(a)$.
Similar argument holds for lines of types $(L6)$ and $(L7)$.

Next, consider a line $\{(a,b,i),(b,c,j),(c,a,k)\}$ of type $(L8)$. We have
$\psi(a,b,i) \psi(b,c,j)=\phi(b)\phi(\bar{a})\phi(c)\phi(\bar{b})\delta(a,b,i)\delta(b,c,j)$.
Since $\{i,j,k\}=\mathbb{Z}_{3}$, $\{\delta(a,b,i),\delta(b,c,j),\delta(c,a,k)\}=\{\delta(a),\delta(b),\delta(c)\}$.
Since $\{a,b,c\}\in S$, Lemma \ref{lem5.2}$(iii)$ implies that $\delta(a,b,i)\delta(b,c,j)=\delta(c,a,k)$. So,
$\psi(a,b,i)\psi(b,c,j) = \phi(b)\phi(c)\phi(\bar{a})\phi(\bar{b})\delta(c,a,k)
=\phi(a)\phi(\bar{c})\delta(c,a,k)=\psi(c,a,k)$. Notice that the second
equality holds since $\{a,b,c\}$ and $\{\bar{a},\bar{b},\bar{c}\}$
are lines of $Y$.

Finally, consider a line $\{(a,u,i),(b,v,j),(c,w,k) \}$ of type
$(L9)$. Here the lines $au,bv,cw$ are in $S$, $j=i+\theta(au,bv)$
and $k=i+\theta(au,cw)$. Let
$\delta(a,u,i)=\delta(\alpha),\delta(b,v,j)=\delta(\beta)$ and
$\delta(c,w,k)=\delta(\gamma)$, where $\alpha\in au, \beta\in bv$
and $\gamma\in cw$. So $\epsilon(\alpha)=i,\epsilon(\beta)=j$ and
$\epsilon(\gamma)=k$. Since
$\epsilon(\beta)-\epsilon(\alpha)=j-i=\theta(au,bv)$, Lemma
\ref{lem5.1} implies that $\alpha\sim \beta$. Similarly,
$\alpha\sim \gamma$. Thus $\{\alpha,\beta,\gamma\}$ is a line of
$Q$ not contained in $S$. Then by Lemma
\ref{lem5.2}$(iii)$,
$\delta(a,u,i)\delta(b,v,j)=\delta(\alpha)\delta(\beta)=\delta(\gamma)\lambda=\delta(c,w,k)\lambda$.
So
$\psi(a,u,i)\psi(b,v,j)=\phi(u)\phi(\bar{a})\phi(v)\phi(\bar{b})\delta(c,w,k)\lambda$.
Since $v$ and $\bar{a}$ are in distance three from each other,
$[\phi(\bar{a}),\phi(v)]=\lambda$. So
$\phi(\bar{a})\phi(v)=\phi(v)\phi(\bar{a})[\phi(\bar{a}),\phi(v)]=\phi(v)\phi(\bar{a})\lambda$.
Then $\psi(a,u,i)\psi(b,v,j)=\phi(u)\phi(v)\phi(\bar{a})\phi(\bar{b})\delta(c,w,k)=\phi(w)\phi(\bar{c})\delta(c,w,k)
=\psi(c,w,k)$. The second equality holds since $\{ \bar{a},\bar{b},\bar{c} \}$ and $\{ u,v,w \}$ are lines of $Y$. This completes the proof.
\end{proof}

\bigskip \noindent {\Large \textbf{Addresses:}}

\bigskip \noindent {\small \textsc{B. De Bruyn.} Department of Pure Mathematics and Computer Algebra, Ghent University, Krijglaan 281 (S22), B-9000 Gent, Belgium, \texttt{bdb@cage.ugent.be}}

\medskip \noindent {\small \textsc{B. K. Sahoo.} School of Mathematical Sciences, National Institute of Science Education and Research, Sainik School Post, Bhubaneswar-751005, Orissa, India, \texttt{bksahoo@niser.ac.in}}

\medskip \noindent {\small \textsc{N. S. N. Sastry.} Statistics and Mathematics Unit, Indian Statistical Institute, Mysore Road, R. V. College Post, Bangalore-560059, India, \texttt{nsastry@isibang.ac.in}}

\end{document}